\documentclass[12pt]{amsart}
\usepackage{amsthm}
\usepackage{amssymb}
\usepackage{cite}
\usepackage{longtable}
\usepackage{geometry}
\usepackage{enumerate}
\usepackage{setspace}
\usepackage{tikz-cd}
\usetikzlibrary{matrix, arrows}
\usepackage[unicode=true]
 {hyperref}
\hypersetup{
colorlinks=true,
urlcolor=black,
citecolor=blue,
linkcolor=blue,
}

\allowdisplaybreaks

\newtheorem{thm}{Theorem}[section]
\newtheorem{prop}[thm]{Proposition}
\newtheorem{lem}[thm]{Lemma}
\newtheorem{cor}[thm]{Corollary}

\theoremstyle{definition}

\theoremstyle{remark}

\numberwithin{equation}{section}

\newcommand{\bZ}{\mathbb{Z}}
\newcommand{\bN}{\mathbb{N}}
\newcommand{\Perm}{\mathrm{Perm}}
\newcommand{\Hol}{\mathrm{Hol}}
\newcommand{\NHol}{\mathrm{NHol}}
\newcommand{\Aut}{\mathrm{Aut}}
\newcommand{\Norm}{\mathrm{Norm}}
\newcommand{\ep}{\epsilon}
\newcommand{\mmod}{\hspace{-3.5mm}\mod}
\newcommand{\pmmod}{\hspace{-2.5mm}\pmod}

\begin{document}

\large 

\title[The multiple holomorph of a semidirect product of groups]{The multiple holomorph of a semidirect product\\ of groups having coprime exponents}

\author{Cindy (Sin Yi) Tsang}
\address{School of Mathematics (Zhuhai), Sun Yat-Sen University, Zhuhai, Guangdong, China}
\email{zengshy26@mail.sysu.edu.cn}\urladdr{http://sites.google.com/site/cindysinyitsang/} 

\date{\today}

\maketitle

\vspace{-5mm}

\begin{abstract}Given any group $G$, the multiple holomorph $\mathrm{NHol}(G)$ is the normalizer of the holomorph $\mathrm{Hol}(G) = \rho(G)\rtimes \mathrm{Aut}(G)$ in the group of all permutations of $G$, where $\rho$ denotes the right regular representation. The quotient $T(G) = \mathrm{NHol}(G)/\mathrm{Hol}G)$ has order a power of $2$ in many of the known cases, but there are exceptions. We shall give a new method of constructing elements (of odd order) in $T(G)$ when $G=A\rtimes C_d$, where $A$ is a group of finite exponent coprime to $d$ and $C_d$ is the cyclic group of order $d$.
\end{abstract}

\tableofcontents

\vspace{-10mm}

\section{Introduction}

Let $G$ be a group and write $\Perm(G)$ for the group of all permutations of $G$. A subgroup $N$ of $\Perm(G)$ is said to be \emph{regular} if its action on $G$ is regular. For example, both $\lambda(G)$ and $\rho(G)$ are regular subgroups of $\Perm(G)$, where 
\[ \begin{cases}
\lambda:G\longrightarrow\Perm(G);\hspace{1em}\lambda(\sigma) = (x\mapsto \sigma x)\\
\rho: G\longrightarrow\Perm(G);\hspace{1em}\rho(\sigma) = (x\mapsto x\sigma^{-1})
\end{cases}\]
are the left and right regular representations of $G$. Recall that the \emph{holomorph} of $G$ is defined to be the subgroup
\[ \Hol(G) = \rho(G)\rtimes \Aut(G)\]
of $\Perm(G)$, or alternatively, it is easy to check that
\[ \Norm_{\Perm(G)}(\lambda(G)) = \Hol(G) = \Norm_{\Perm(G)}(\rho(G)).\]
The \emph{multiple holomorph} of $G$ in turn is defined to be the normalizer
\[ \NHol(G) = \Norm_{\Perm(G)}(\Hol(G)).\]
We are interested in the quotient group
\[ T(G) = \frac{\NHol(G)}{\Hol(G)}.\]
The study of $T(G)$ was initiated by G. A. Miller \cite{Miller}, and was motivated by the fact that $T(G)$ acts regularly (via conjugation) on the regular subgroups $N$ of $\Perm(G)$ for which $N\simeq G$ and $\Norm_{\Perm(G)}(N) = \Hol(G)$. Let us note that $\pi_{-1}$ always lies in $\NHol(G)$, where
\[ \pi_{-1}:G\longrightarrow G;\hspace{1em}\pi_{-1}(x) = x^{-1}\]
is the inverse map. The subgroup $\langle\pi_{-1}\Hol(G)\rangle$ has order $1$ or $2$, and its elements correspond to the regular subgroups $\lambda(G),\rho(G)$. Clearly $\langle\pi_{-1}\Hol(G)\rangle$ is trivial precisely when $G$ is abelian. This corresponds to the fact that $\lambda(G)$ and $\rho(G)$ coincide exactly when $G$ is abelian.

\vspace{1mm}

The structure of $T(G)$ has been computed for various families of groups $G$; see \cite{Caranti1,Caranti2,Caranti3,Kohl NHol,Mills,Miller,Tsang NHol,Tsang squarefree}. In many of the known cases $T(G)$ turns out to have order a power of $2$. It is in fact elementary $2$-abelian for the following groups $G$.
\begin{enumerate}[$\bullet$]
\item Finitely generated abelian groups $G$; see \cite{Miller,Mills,Caranti1}.
\item Finite dihedral or dicyclic groups $G$; see \cite{Kohl NHol}.
\item Finite perfect groups $G$ with trivial center; see \cite{Caranti2}.
\item Finite quasisimple or almost simple groups $G$; see \cite{Tsang NHol}.
\item Finite groups $G$ of squarefree order; see \cite{Tsang squarefree}.
\end{enumerate}
Nevertheless, there are groups $G$ for which $T(G)$ has elements of odd order, as shown by A. Caranti \cite{Caranti3}, whose result was recently extended by the author in \cite{Tsang squarefree}. More specifically, observe that for each $\ell\in\bZ$, the $\ell$th power map
\[ \pi_\ell: G\longrightarrow G;\hspace{1em}\pi_\ell(x) = x^\ell\]
is bijective when $G$ has finite exponent coprime to $\ell$. In general $\pi_\ell$ need not lie in $\NHol(G)$, unless $\ell\equiv\pm1$ (mod $\exp(G)$). But under suitable conditions, it does lie in $\NHol(G)$, and  $\pi_\ell\Hol(G)$ defines an element of $T(G)$. The next result is \cite[Theorem 3.1]{Tsang squarefree}, which generalizes \cite[Proposition 3.1]{Caranti3}.

\begin{thm}\label{old thm}Let $G$ be a finite $p$-group of nilpotency class $2\leq n_c\leq p-1$. Let $Z(G)$ and $\gamma_3(G)$, respectively, denote its center and the third term in its lower central series. Let $r\geq1$ and $t\geq0$ be the integers such that
\[ \exp(G/Z(G)) = p^r\mbox{ and }\exp(\gamma_3(G)) = p^t.\]
Then, for all $\ell\in\bZ$ coprime to $p$ with $\ell\equiv1\pmod{p^t}$, we have $\pi_\ell\in\NHol(G)$, and $\pi_\ell\Hol(G)$ is trivial in $T(G)$ exactly when $\ell\equiv1\pmod{p^r}$.\end{thm}

Theorem~\ref{old thm} may be used to produce examples of finite $p$-groups $G$ and in particular nilpotent groups $G$ for which $T(G)$ has elements of odd order. As noted in \cite[p. 4]{Tsang squarefree}, there are non-nilpotent (but solvable) groups $G$ for which $T(G)$ has elements of odd order. By \cite[Example 3.7]{Tsang squarefree}, there is a group $G$ of order $3^6$ for which $T(G)$ has order $18$, but the $\pi_\ell$ for $\ell$ coprime to $3$ only give rise to a subgroup of order $2$ in $T(G)$. This means that the elements of odd order in $T(G)$ do not always come from these power maps $\pi_\ell$.


\vspace{1mm}

In this paper, we give a new method of constructing elements in $T(G)$, and as an application, give a simple way of creating non-nilpotent (but solvable) groups $G$ for which $T(G)$ has elements of odd order. More specifically, take $G = A\rtimes C_d$, where $A$ is a group and $C_d$ is the cyclic group of order $d$. Let
\[d = p_1^{n_1}\cdots p_r^{n_r},\mbox{ where $p_1,\dots,p_r$ are distinct and }n_1,\dots,n_r\in\bN,\]
be the prime factorization of $d$. Then, alternatively, we have
\begin{equation}\label{G def} G = A\rtimes \left(\langle\tau_1\rangle\times\cdots\times\langle\tau_r\rangle\right),\end{equation}
where for each $1\leq s\leq r$, the element $\tau_s$ has order $p_s^{n_s}$, and write $\psi_s$ for the automorphism on $A$ such that $\tau_s a\tau_s^{-1} = \psi_s(a)$ for all $a\in A$. Notice that $\psi_s$ has order dividing $p_s^{n_s}$. Let us also fix a number $k_s = 1 + u_sp_s^{m_s}$, where
\[ u_s,m_s\in\bN\mbox{ are such that $p_s\nmid u_s$ and $m_s\leq n_s$, with $m_s\geq2$ if $p_s=2$}.\]
For $x\in\bN$ and $i\in\bN_{\geq0}$, define
\[ S(x,i) = 1+x+\cdots+x^{i-1},\mbox{ with }S(x,0) = 0.\]
Then, as we shall see in Proposition~\ref{bijective prop}, the map
\[ \pi : G\longrightarrow G;\hspace{1em}\pi(a\tau_1^{i_1}\cdots\tau_r^{i_r}) = a\tau_1^{S(k_1,i_1)}\cdots \tau_r^{S(k_r,i_r)},\]
where $a\in A$ and the $i_1,\dots,i_r$ are taken to be non-negative, is a well-defined bijection. For $\theta\in\Aut(G)$ and for each $1\leq s\leq r$, clearly we may write
\begin{equation}\label{theta}\theta(\tau_s) = c_s\tau_s^{\theta_s},\mbox{ where $c_s\in A$ and $\theta_s\in\bZ$}. \end{equation}
Let us define $\ep_s = 1$ if $p_s=2$, and $\ep_s=0$ if $p_s$ is odd. We shall prove:

\begin{thm}\label{thm}Let $G$ be as in $(\ref{G def})$. In the above notation, if
\begin{enumerate}[(1)]
\item $A$ has finite exponent coprime to $p_1,\dots, p_r$,
\item $\psi_s$ has order dividing $p_s^{m_s}$ for all $1\leq s\leq r$,
\item $\theta_s\equiv1\pmod{p_s^{n_s-m_s+\ep_s}}$ for all $\theta\in\Aut(G)$ and $1\leq s\leq r$,
\end{enumerate}
then $\pi\in\NHol(G)$. If in addition
\begin{enumerate}[(1)]\setcounter{enumi}{+3}
\item $n_s\leq 2m_s-\ep_s$ for all $1\leq s\leq r$,
\end{enumerate}
then $\pi\Hol(G)$ has order $p_1^{n_1-m_1}\cdots p_r^{n_r-m_r}$ in $T(G)$.
\end{thm}

Conditions (1),(2),(4) are easy to check. However, in general it is not clear whether condition (3) is satisfied. In Section~\ref{sec app}, given condition (1), and that $A$ is abelian and $\psi_1,\dots,\psi_r$ lie in the center of $\Aut(A)$, we shall give a simple way of picking $m_1,\dots,m_r$ such that all of conditions (2) -- (4) hold. Taking $\psi_1,\dots\psi_r$ to be power maps, our method allows us to easily produce examples of $G$ for which $T(G)$ has elements of odd order.

\vspace{1mm}

Finally, we remark that Theorem~\ref{thm} does not explain why $T(G)$ contains elements of odd order for the groups $G$ in \cite[(1.5) and Example 3.7]{Tsang squarefree}, except for the group $G$ with \textsc{SmallGroup ID} equal to $(63,1)$.

\section{Preliminary modular arithmetic}\label{sec mod}

In this section, let $p$ be a prime, and fix a number $k = 1 + up^m$, where
\[ u,m\in\bN\mbox{ are such that $p\nmid u$, with $m\geq 2$ if $p=2$}.\]
For any $i\in\bN$, by the binomial theorem, we have
\begin{equation}\label{S expression} S(k,i) = \frac{k^i-1}{k-1} = \frac{1}{up^m}\sum_{t=1}^{i}{i\choose t}u^t p^{mt} = i + \sum_{t=1}^{i-1}{i\choose t+1}u^t p^{mt}.\end{equation}
For each $t\in\bN$, let $p^{x_t}$ be the exact power of $p$ dividing $(t+1)!$. It is known that we have the explicit formula
\begin{equation}\label{xt} x_t = \left\lfloor \frac{t+1}{p}\right\rfloor + \left\lfloor \frac{t+1}{p^2}\right\rfloor  + \left\lfloor \frac{t+1}{p^3}\right\rfloor  + \cdots \mbox{ and so } x_t < \frac{t+1}{p-1}.\end{equation}
Also, put $\ep=1$ if $p=2$, and $\ep=0$ if $p$ is odd. We shall require the following properties of $S(k,i)$ modulo powers of $p$.

\begin{lem}\label{S lem}For any $i\in\bN$, the exact powers of $p$ dividing $i$ and $S(k,i)$ are the same.
\end{lem}
\begin{proof}Let $p^\ell$ exactly divide $i$. For each $1\leq t\leq i-1$, we have
\[ \left(\frac{i(i-1)\cdots (i-t)}{(t+1)!}\right)p^{mt} \equiv0\pmmod{p^{\ell + mt - x_t}},\]
and by (\ref{S expression}), it then suffices to show that $mt - x_t\geq 1$. For $t=1,2$, clearly this holds, because $m\geq 2$ if $p=2$. For $3\leq t\leq i-1$, by (\ref{xt}) we have
\[ mt - x_t \geq \begin{cases} mt - (t+1) = t(m-1) - 1 \geq 3 -1 = 2 &\mbox{if $p=2$},\\
 mt - \frac{t+1}{2} = t\left(m -\frac{1}{2}\right) - \frac{1}{2} \geq \frac{3}{2} - \frac{1}{2} = 1&\mbox{if $p$ is odd}.\end{cases}\]
This proves the claim.
\end{proof}

\begin{prop}\label{bijective prop}For any $n\in\bN$, the map
\[ \Xi : \bZ/p^n\bZ \longrightarrow \bZ/p^n\bZ;\hspace{1em}\Xi(i\mmod{p^n}) = S(k,i)\mmod{p^n},\]
where $i$ is taken to be non-negative, is a well-defined bijection. Moreover:
\begin{enumerate}[(a)]
\item If $n\geq m$, then for all $i\equiv1\pmod{p^{n-m+\ep}}$, we have
\[ \Xi(i\mmod{p^n}) = i\mmod{p^n}.\]
\item If $n\leq 2m - \ep$, then for all $\ell,i\in\bN_{\geq0}$, we have
\[ \Xi^\ell(i\mmod{p^n}) = S(k^\ell,i)\mmod{p^n}.\]
\end{enumerate}
\end{prop}
\begin{proof}For any $i,j,\ell\in\bN_{\geq0}$, with $j\geq i$ say, we have the identities
\begin{align*} S(k,i+\ell p^n) & = S(k,i) + k^iS(k,\ell p^n),\\
 S(k,j) - S(k,i) &= k^iS(k,j-i).
\end{align*}
From Lemma~\ref{S lem}, we then see that $\Xi$ is well-defined and injective. Since the group $\bZ/p^n\bZ$ is finite, it follows that $\Xi$ is in fact a bijection.

\vspace{1mm}

To prove (a), suppose that $n\geq m$ and $i\equiv1$ (mod $p^{n-m+\ep}$). Then, for each $1\leq t\leq i-1$, we have
\[\left(\frac{i(i-1)\cdots (i-t)}{(t+1)!}\right)p^{mt} \equiv0\pmmod{p^{n-m+\ep + mt - x_t}},\]
and by (\ref{S expression}), it is enough to show that $m(t-1)-x_t+\ep\geq 0$. For $t=1,2$, this is clear. For $3\leq t\leq i-1$, by (\ref{xt}) we have
\begin{align*}
&\hspace{10mm}m(t-1) - x_t + \ep\\ &\geq \begin{cases}
m(t-1) - (t+1) + 1 = t(m-1) - m \geq 2m - 3 \geq 1 &\mbox{if $p=2$},\\
m(t-1) - \frac{t+1}{2} = t\left(m-\frac{1}{2}\right) - m - \frac{1}{2} \geq 2m - 2 \geq 0&\mbox{if $p$ is odd}.
\end{cases}\end{align*}
Here, we used the assumption that $m\geq 2$ if $p=2$. Thus, indeed (a) holds.

\vspace{1mm}

To prove (b), suppose that $n\leq 2m-\ep$. Then $k^\ell \equiv 1 + \ell up^m$ (mod $p^n$) by the binomial theorem, and so we have
\[S(k^\ell,i) \equiv  i_\ell \pmmod{p^n},\mbox{ where }i_\ell = i + {i\choose 2}\ell u p^m.\]
For $\ell=0,1$, the claim is trivial. For $\ell\geq 1$, the claim for $\ell$ implies that
\[ \Xi^{\ell + 1}(i\mmod{p^n}) 
 = \Xi(i_\ell \mmod{p^n}) = i_\ell + {i_\ell \choose 2}up^m\mmod{p^n}.\]
But observe that
\[ {i_\ell \choose 2}up^m = {i\choose 2}up^m + \frac{1}{2}(2i-1){i\choose 2}\ell u^2p^{2m} + \frac{1}{2}{i\choose 2}^2\ell^2u^3p^{3m}.\]
Since $n\leq 2m-\ep$, the last three terms are zero modulo $p^{n}$, and so
\[ \Xi^{\ell+1}(i\mmod{p^n}) = i_\ell + {i\choose 2}up^m \mmod{p^n} = i + {i\choose 2}(\ell+1)up^m \mmod{p^n}.\]
We have thus proven (b) by induction.\end{proof}

\begin{lem}\label{tau lem}For any $n\in\bN$ and $i\in\bN_{\geq0}$, we have
\[ S(k,\widetilde{i}-1) - S(k,p^n-1) \equiv k'i \pmmod{p^n},\]
where $\widetilde{i},k'\in\bN$ are such that
\[ S(k,\widetilde{i})\equiv i\pmmod{p^n}\mbox{ and }kk'\equiv1\pmmod{p^n}.\]
\end{lem}
\begin{proof}Without loss of generality, we may take $1\leq\widetilde{i}\leq p^n$. Then, we have
\[ S(k,\widetilde{i}-1) - S(k,p^n-1) = -k^{\widetilde{i}-1}S(k,p^n-\widetilde{i}).\]
Since $S(k,p^n)\equiv 0$ (mod $p^n$) by Lemma~\ref{S lem}, we deduce that
\[ -k^{\widetilde{i}-1}S(k,p^n-\widetilde{i}) \equiv -k^{\widetilde{i}-1}\cdot - k^{p^n-\widetilde{i}}S(k,\widetilde{i}) \equiv k^{p^n-1} i \pmmod{p^n}.\]
But clearly $k^{p^n}\equiv1$ (mod $p^n$) and so the claim follows.
\end{proof}

\section{Proof of Theorem~\ref{thm}}\label{sec construct}

In this section, suppose that $G$ is as in (\ref{G def}), and we shall use the notation defined prior to Theorem~\ref{thm}.

\subsection{First claim} By definition, we have
\[\Hol(G) = \rho(G)\rtimes\Aut(G)\mbox{ and }\NHol(G) = \Norm_{\Perm(G)}(\Hol(G)).\]
Under conditions (1) -- (3) in Theorem~\ref{thm}, we then deduce from Lemmas~\ref{thm lem} and~\ref{thm lem'} below that $\pi\in\NHol(G)$, as desired.

\begin{lem}\label{thm lem}Suppose that condition $(2)$ in Theorem~$\ref{thm}$ holds. Then:
\begin{enumerate}[(a)]
\item For any $b\in A$, we have $\pi\rho(b)\pi^{-1} = \rho(b)$.
\item For any $1\leq s\leq r$, we have $\pi\rho(\tau_s)\pi^{-1}\in \rho(\tau_s^{S(k_s,p_s^{n_s}-1)})^{-1}\cdot \Aut(G)$.
\end{enumerate}
\end{lem}
\begin{proof}[Proof of (a)]Let $b\in A$. For any $a\in A$ and $i_1,\dots,i_r\in\bN_{\geq0}$, we have
\begin{align*}
(\pi \rho(b))(a\tau_1^{i_1}\cdots\tau_r^{i_r}) & =  a(\psi_1^{i_1}\cdots\psi_r^{i_r})(b)^{-1}\cdot \tau_1^{S(k_1,i_1)}\cdots\tau_r^{S(k_r,i_r)},\\
(\rho(b)\pi)(a\tau_1^{i_1}\cdots\tau_r^{i_r}) & = a(\psi_1^{S(k_1,i_1)}\cdots\psi_r^{S(k_r,i_r)})(b)^{-1}\cdot\tau_1^{S(k_1,i_1)}\cdots\tau_r^{S(k_r,i_r)}.
\end{align*}
For each $1\leq s\leq r$, note that $S(k_s,i_s) \equiv  i_s$ (mod $p_s^{m_s}$), and since $\psi_s^{p_s^{m_s}} = \mathrm{Id}_A$ by condition (2), we have $\psi_s^{S(k_s,i_s)} = \psi_s^{i_s}$. Hence, indeed $\pi\rho(b) = \rho(b)\pi$.\end{proof}
\begin{proof}[Proof of (b)] Let $1\leq s\leq r$, and for brevity, put
\[ \theta = \rho(\tau_s^{S(k_s,p_s^{n_s}-1)})\pi\rho(\tau_s)\pi^{-1}.\]
For any $a\in A$ and $i_1,\dots,i_r\in\bN_{\geq 0}$, we have
\[\theta(a\tau_1^{i_1}\cdots\tau_r^{i_r}) 
 = a\tau_1^{i_1}\cdots\tau_s^{S(k_s,\widetilde{i}_s-1)-S(k_s,p_s^{n_s}-1)}\cdots \tau_r^{i_r}
 = a\tau_1^{i_1}\cdots\tau_s^{k'_si_s}\cdots \tau_r^{i_r}\]
by Lemma~\ref{tau lem}, where $\widetilde{i}_s,k_s'\in\bN$ are such that
\[ S(k_s,\widetilde{i}_s)\equiv i_s\pmmod{p_s^{n_s}}\mbox{ and }k_sk'_s\equiv1\pmmod{p_s^{n_s}}.\]
For any $b\in A$ and $j_1,\dots,j_r\in\bN_{\geq0}$, observe that
\begin{align*}
&\hspace{7.5mm}\theta(a\tau_1^{i_1}\cdots\tau_r^{i_r}\cdot b\tau_1^{j_1}\cdots\tau_r^{j_r}) \\ &= a(\psi_1^{i_1}\cdots\psi_s^{i_s}\cdots\psi_r^{i_r})(b)\cdot\tau_1^{i_1+j_1}\cdots\tau_s^{k'_s(i_s+j_s)}\cdots\tau_r^{i_r+j_r},\\
&\hspace{7.5mm}\theta(a\tau_1^{i_1}\cdots\tau_r^{i_r})\theta(b\tau_1^{j_1}\cdots\tau_r^{j_r})\\&=a(\psi_1^{i_1}\cdots\psi_s^{k'_si_s}\cdots\psi_r^{i_r})(b)\cdot\tau_1^{i_1+j_1}\cdots \tau_s^{k'_s(i_s+j_s)}\cdots \tau_r^{i_r+j_r}.\end{align*}
Note that $k'_s\equiv1$ (mod $p_s^{m_s}$), and since $\psi_s^{p_s^{m_s}} = \mathrm{Id}_A$ by condition (2), we have $\psi_s^{k_s'} = \psi_s$. The above then shows that $\theta\in \Aut(G)$, as desired.
\end{proof}

\begin{lem}\label{thm lem'}Suppose that conditions $(1)$ -- $(3)$ in Theorem~$\ref{thm}$ hold. Then, for any $\theta\in \Aut(G)$, we have $\pi\theta\pi^{-1} =\theta$.
\end{lem}
\begin{proof}Let $\theta\in\Aut(G)$. For each $1\leq s\leq r$, write $\theta(\tau_s) = c_s\tau_s^{\theta_s}$ as in (\ref{theta}), where $c_s\in A$ and $\theta_s\in\bZ$. Note that since $\tau_s$ has order $p_s^{n_s}$, and $\psi_s^{p_s^{m_s}} =\mathrm{Id}_A$ by condition (2), we see that
\[ 1_G = (c_s\tau_s^{\theta_s})^{p_s^{n_s}} = (c_s\psi_s^{\theta_s}(c_s)\cdots \psi_s^{\theta_s(p_s^{m_s}-1)}(c_s))^{p_s^{n_s-m_s}}.\]
By condition (1), this implies that in fact
\begin{equation}\label{cs}c_s\psi_s^{\theta_s}(c_s)\cdots \psi_s^{\theta_s(p_s^{m_s}-1)}(c_s) = 1_G.\end{equation}
For any $a\in A$ and $i_1,\dots,i_r\in\bN_{\geq0}$, observe that $\theta(a)\in A$ by condition (1), and so we compute that
\begin{align*}
(\pi\theta)(a\tau_1^{i_1}\cdots\tau_r^{i_r}) & = \pi(\theta(a)(c_1\tau_1^{\theta_1})^{i_1}\cdots(c_r\tau_r^{\theta_r})^{i_r})\\
& = \pi(\theta(a)\mathfrak{c}_1\tau_1^{\theta_1i_1}\cdots \mathfrak{c}_r\tau_r^{\theta_ri_r})\\
& = \theta(a)\mathfrak{c}_1\psi_1^{\theta_1i_1}(\mathfrak{c}_2)\cdots(\psi_1^{\theta_1i_1}\cdots\psi_{r-1}^{\theta_{r-1}i_{r-1}})(\mathfrak{c}_{r})\\ & \hspace{4.5cm}\cdot\tau_1^{S(k_1,\theta_1i_1)}\cdots\tau_r^{S(k_r,\theta_ri_r)},\\
(\theta\pi)(a\tau_1^{i_1}\cdots\tau_r^{i_r}) &  = \theta(a)(c_1\tau_1^{\theta_1})^{S(k_1,i_1)}\cdots (c_r\tau_r^{\theta_r})^{S(k_r,i_r)}\\
& = \theta(a)\mathfrak{c}_1'\tau_1^{\theta_1S(k_1,i_1)}\cdots \mathfrak{c}_r'\tau_r^{\theta_rS(k_r,i_r)}\\
& = \theta(a)\mathfrak{c}_1'\psi_1^{\theta_1S(k_1,i_1)}(\mathfrak{c}_2')\cdots (\psi_1^{\theta_1S(k_1,i_1)}\cdots\psi_{r-1}^{\theta_{r-1}S(k_{r-1},i_{r-1})})(\mathfrak{c}_r')\\
& \hspace{7.75cm}\cdot\tau_1^{\theta_1S(k_1,i_1)}\cdots\tau_r^{\theta_rS(k_r,i_r)}.
\end{align*}
Here, for each $1\leq s\leq r$, we define
\begin{align*}\mathfrak{c}_s &= c_s\psi_s^{\theta_s}(c_s)\cdots \psi_s^{\theta_s(i_s-1)}(c_s),\\
\mathfrak{c}_s' & = c_s\psi_s^{\theta_s}(c_s)\cdots \psi_s^{\theta_s(S(k_s,i_s)-1)}(c_s).\end{align*}
But $S(k_s,i_s)\equiv i_s$ (mod $p_s^{m_s}$), and $\psi_s^{p_s^{m_s}}=\mathrm{Id}_A$ by condition (2). So, we have
\[ \psi_s^{\theta_sS(k_s,i_s)} = \psi_s^{\theta_si_s},\mbox{ and }\mathfrak{c}_s = \mathfrak{c}_s'\]
because of (\ref{cs}). Also, by Proposition~\ref{bijective prop}(a) and condition (3), we have
\[ S(k_s,\theta_si_s) \equiv S(k_s^{i_s},\theta_s)S(k_s,i_s) \equiv \theta_sS(k_s,i_s)\pmmod{p_s^{n_s}}.\]
We then deduce that $\pi\theta=\theta\pi$, as desired.
\end{proof}

\subsection{Second claim} 

Observe that 
\[ \pi(1_G) =1_G,\, \pi(a) = a \mbox{ for all }a\in A,\, \pi(\tau_s) = \tau_s \mbox{ for all }1\leq s\leq r.\]
For any $\ell\in\bN$, we then deduce that
\[ \pi^\ell \in \Hol(G) \iff \pi^\ell \in \Aut(G) \iff \pi^\ell = \mathrm{Id}_G.\]
Suppose now that condition (4) in Theorem~\ref{thm} holds. By Proposition~\ref{bijective prop}(b), we know that explicitly 
\[ \pi^\ell : G\longrightarrow G;\hspace{1em}\pi^\ell(a\tau_1^{i_1}\cdots\tau_r^{i_r}) = a\tau_1^{S(k_1^\ell,i_1)}\cdots \tau_r^{S(k_r^\ell,i_r)},\]
where $a\in A$ and the $i_1,\dots,i_r$ are taken to be non-negative. Hence, it follows that $\pi^\ell \in\Hol(G)$ if and only if for each $1\leq s\leq r$, we have
\begin{equation}\label{cong} S(k_s^\ell,i_s) \equiv i_s\pmmod{p_s^{n_s}}\mbox{ for all }i_s\in\bN_{\geq0}.\end{equation}
The order of $k_s$ mod $p_s^{n_s}$ is equal to $p_s^{n_s-m_s}$; this is an elementary arithmetic fact and requires the assumption that $m_s\geq 2$ if $p_s=2$. Thus, if $\ell$ is divisible by $p_s^{n_s-m_s}$, then $k_s^\ell\equiv1$ (mod $p_s^{n_s}$) and (\ref{cong}) clearly holds. Conversely, if (\ref{cong}) holds, then by taking $i_s=2$, we obtain $1+k_s^{\ell}\equiv 2$ (mod $p_s^{n_s})$, which implies that $p_s^{n_s-m_s}$ divides $\ell$. It then follows that $\ell= p_1^{n_1-m_1}\cdots p_r^{n_r-m_r}$ is the smallest natural number for which $\pi^\ell \in\Hol(G)$. This proves the claim.

\section{Construction of examples}\label{sec app}

In this section, suppose that $G$ is as in (\ref{G def}), and we shall use the notation defined prior to Theorem~\ref{thm}. For each $1\leq s\leq r$, also let $0\leq y_s\leq n_s$ be the integer such that $\psi_s$ has order $p_s^{y_s}$. 

\begin{prop}\label{criterion}Let $\theta\in \Aut(G)$. For each $1\leq r\leq s$, write $\theta(\tau_s) = c_s\tau_s^{\theta_s}$ as in $(\ref{theta})$, where $c_s\in A$ and $\theta_s\in \bZ$. If $\theta(A)= A$, then 
\[ \theta|_A\psi_s = \lambda(c_s)\rho(c_s)\psi_s^{\theta_s}\theta|_A\mbox{ in }\Aut(A).\]
If in addition $A$ is abelian and $\psi_s$ commutes with $\theta|_A$, then
\[ \psi_s^{\theta_s-1}=\mathrm{Id}_A\mbox{ and in particular }\theta_s\equiv1\pmmod{p_s^{y_s}}.\]
\end{prop}
\begin{proof} For any $a\in A$, we have $\tau_sa\tau_s^{-1}=\psi_s(a)$. Hence, if $\theta(A)= A$, then
\[ (\theta\psi_s)(a) = c_s\tau_s^{\theta_s}\theta(a)\tau_s^{-\theta_s}c_s^{-1} = (\lambda(c_s)\rho(c_s)\psi_s^{\theta_s}\theta)(a),\]
which yields the claimed equality in $\Aut(A)$. If in addition $A$ is abelian and $\psi_s$ commutes with $\theta|_A$, then
\[ \psi_s\theta|_A = \theta|_A\psi_s = \lambda(c_s)\rho(c_s)\psi_s^{\theta_s}\theta|_A = \psi_s^{\theta_s}\theta|_A,\]
which implies $\psi_s^{\theta_s-1} = \mathrm{Id}_A$, as desired.
\end{proof}

Condition (1) in Theorem~\ref{thm} implies $\theta(A)= A$ for all $\theta\in \Aut(A)$. From Proposition~\ref{criterion}, we then immediately deduce the following corollary.

\begin{cor}\label{cor} Suppose that condition $(1)$ in Theorem~$\ref{thm}$ holds, that $A$ is abelian, and $\psi_1,\dots,\psi_r$ lie in the center of $\Aut(A)$. If $n_s - m_s + \ep_s \leq y_s \leq m_s$ for all $1\leq s\leq r$, then conditions $(2)$ -- $(4)$ in Theorem~$\ref{thm}$ also hold. 
\end{cor}

In view of Corollary~\ref{cor}, we may construct groups $G$ of the form (\ref{G def}) and choose $m_1,\dots,m_r$ for which all of conditions (1) -- (4) in Theorem~\ref{thm} hold, as follows.
\begin{enumerate}[1.]
\item Start with an abelian group $A$ of finite exponent $\exp(A)$.
\item Find $\psi_1,\dots,\psi_r$ in the center of $\Aut(A)$ of prime power orders $p_1^{y_1},\dots,p_r^{y_r}$, such that $p_1,\dots,p_r$ are distinct and none of them divides $\exp(A)$.
\item For each $1\leq s\leq r$, find $n_s,m_s\in\bN$ with $n_s-m_s+\ep_s\leq y_s\leq m_s\leq n_s$.
\item Take $d = p_1^{n_1}\cdots p_r^{n_r}$ and define $G$ to be as in (\ref{G def}).
\end{enumerate}
Then, from Corollary~\ref{cor}, we deduce that all of conditions (1) -- (4) in Theorem~\ref{thm} are satisfied. In particular, for the group $G$ defined as in (\ref{G def}), the quotient $T(G)$ has an element of order $p_1^{n_1-m_1}\cdots p_{r}^{n_r-m_r}$.

\vspace{1mm}

For each $\ell\in \bZ$, the $\ell$th power map 
\[ \pi_\ell : A\longrightarrow A;\hspace{1em}\pi_\ell(a) = a^\ell\]
on $A$ is a homomorphism when $A$ is abelian. Hence, for any abelian group $A$ of finite exponent $\exp(A)$, we have the subgroup
\[ \left\{\pi_\ell : \ell\in\bZ\mbox{ with }\gcd(\ell,\exp(A))=1\right\} \simeq \left(\frac{\bZ}{\exp(A)\bZ}\right)^\times \]
lying in the center of $\Aut(A)$. By taking $\psi_1,\dots,\psi_r$ to be these power maps in step 2, the method above then allows us to easily construct many examples of $G$ for which $T(G)$ has elements of odd order.

\section*{Acknowledgments} Research supported by ``the Fundamental Research Funds for the Central Universities'' (Award No.: 19lpgy247). The author thanks the editor and the referee for helpful comments.

\end{document}